\documentclass[10pt]{amsart}
\usepackage{amsmath}
\usepackage{amsfonts}
\usepackage{amssymb}
\usepackage{graphicx}
\usepackage{amsthm,graphicx,color,yfonts}
\usepackage{pdfsync}
\usepackage{epstopdf}%
\usepackage{pifont}
\usepackage{bbm}

\usepackage[colorlinks=true]{hyperref}
\hypersetup{linkcolor=red,citecolor=blue,filecolor=dullmagenta,urlcolor=darkblue} 

\newcommand{\YM}{\color{black}}

\newcommand{\RR}{{{\mathbb R}}}

\newcommand{\R} {\mathbb R}
\newcommand{\cuad}{{\sqcap\kern-.68em\sqcup}}

\newcommand{\ve}{\varepsilon}

\newcommand{\be}{\begin{equation}}
\newcommand{\ee}{\end{equation}}

\definecolor{darkgreen}{rgb}{0.2,0.7,0.1}

\newcommand{\sech}{\mathop{\mbox{\normalfont sech}}\nolimits}

\newcommand{\al}{\alpha}
\newcommand{\bt}{\beta}

\def\bm{\left( \begin{array}{cc}}
\def\endm{\end{array}\right)}

\newcommand{\ba}{\begin{equation*}}
\newcommand{\ea}{\begin{equation*}}
\newcommand{\bea}{\begin{eqnarray}}
\newcommand{\eea}{\end{eqnarray}}
\newcommand{\bee}{\begin{eqnarray*}}
\newcommand{\eee}{\end{eqnarray*}}
\newcommand{\ben}{\begin{enumerate}}
\newcommand{\een}{\end{enumerate}}

\def\mm{m}
\numberwithin{equation}{section}

\newtheorem{theorem}{Theorem}[section]
\newtheorem*{theorem*}{Theorem}
\newtheorem{proposition}{Proposition}[section]

\newtheorem{lemma}{Lemma}[section]

\theoremstyle{remark}

\title[Nonexistence of small, odd breathers]{Nonexistence of small, odd breathers for a class of nonlinear wave equations}

\author{Micha{\l}  Kowalczyk}
\address{Departamento de Ingenier\'{\i}a Matem\'atica and Centro
de Modelamiento Matem\'atico (UMI 2807 CNRS), Universidad de Chile, Casilla
170 Correo 3, Santiago, Chile.}
\email {kowalczy@dim.uchile.cl}
\thanks{M. Kowalczyk was partially supported by Chilean research grants FONDECYT 1130126, Fondo Basal CMM-Chile.}
\author{Yvan Martel}
\address{CMLS, \'Ecole polytechnique, CNRS, Universit\'e Paris Saclay, 91128 Palaiseau Cedex, France}
\email{yvan.martel@polytechnique.edu}
\author{Claudio Mu\~noz}
\address{CNRS and Departamento de Ingenier\'{\i}a Matem\'atica and Centro
de Modelamiento Matem\'atico (UMI 2807 CNRS), Universidad de Chile, Casilla
170 Correo 3, Santiago, Chile.}
\email{claudio.munoz@math.u-psud.fr, cmunoz@dim.uchile.cl}
\thanks{C. Mu\~noz was partly funded by  Chilean research grants FONDECYT  1150202, Fondo Basal CMM-Chile, and Millennium
Nucleus Center for Analysis of PDE NC130017}
\thanks{Part of this work was done while the authors were hosted at the IHES during the program \emph{Nonlinear Waves 2016}. Their stay was partly funded by the grant ERC 291214 BLOWDISOL}


\begin{document}

\begin{abstract}
In this note, we show that for a large class of nonlinear wave equations with {\it odd} nonlinearities, any globally defined odd solution which is small in the energy space    decays to $0$ in the local energy norm. In particular, this result shows nonexistence of small, odd breathers for some classical nonlinear Klein Gordon equations such as the sine Gordon equation and $\phi^4$ and $\phi^6$ models. It also partially answers a question of Soffer and Weinstein in \cite[p. 19]{MR1681113} about nonexistence of breathers for the cubic NLKG in dimension one. 
\end{abstract}

\maketitle

\section{Introduction}

{\YM\subsection{Main result}
In this note, we consider the following class of nonlinear wave equations,
\begin{equation}\label{wave ac}
\partial_t^2 u  - \partial_x^2 u = \mm u +f(u), \quad (t,x)\in \RR\times\RR,
\end{equation}
where $\mm\in \R$ and the nonlinearity $f\colon \R\to \R$ is a $C^1$, \emph{odd} function such that for some $p>1$,
\[
 |f'(u)|\lesssim |u|^{p-1}, \qquad   \forall |u|<1.
\]
Denote $F(u)=\int_0^u f$.
An important property of (\ref{wave ac}) is the conservation of \emph{energy}  
\be\label{Energy0}
E(u,\partial_t u):=\int \frac{1}{2}\partial_t u^2+\frac{1}{2}\partial_x u^2  - \frac{\mm}2 u^2- F(u),
\ee
along the flow.  In particular,  $H^1\times L^2$ perturbations of the zero solution are referred as \emph{perturbations in the energy space}. For many standard examples,  the Cauchy problem for the model \eqref{wave ac} is globally well-posed
for  initial data 
$(u(0),\partial_t u(0)) = (u_1^{in},u_2^{in}) \in H^1\times L^2$ small enough. 
Here, we do not make any {\it a priori} assumption  on the sign of $\mm$ nor the nature of $F(u)$ that would guarantee that $E(u)$ is coercive and controls the $H^1\times L^2$ norm of some solutions. Instead, for general nonlinearity $f$ as above, we consider global solutions whose energy norm is uniformly small in time.   
Another important property of \eqref{wave ac}  is that for odd initial data the associated solution is also odd for all time. For the rest of this paper, we   work in such framework, and we   consider only odd perturbations in the energy space. Note  that \eqref{wave ac} is invariant under  space translation  and under the Lorentz transformation but since we consider only odd solutions, these invariances are irrelevant here. 

Set
\be\label{eqvarphi0}
  u_1 =u ,\quad   u_2= \partial_t u   ,
\ee
so that in terms of $(u_1, u_2)$, equation  (\ref{wave ac}) becomes 
\begin{equation}\label{eqvarphi}\left\{\begin{aligned}
&  \partial_t u_1 = u_2 \\  &  \partial_t u_2 = \partial_{x}^2 u_1 +\mm u_1 + f(u_1). \end{aligned}\right.\end{equation}
 
Our main result is the following property for small, odd solutions of (\ref{eqvarphi}).

\begin{theorem}\label{TH1}
There exists  $\ve>0$ such that  any odd global $H^1\times L^2$ solution  $(u_1, u_2)$ of (\ref{eqvarphi})  such that \[
\sup_{t\geq 0}\|(u_1(t),u_2(t))\|_{H^1\times L^2}< \ve,
\]
satisfies
\be\label{Conclusion_0}
\lim_{t \to +\infty}   \|(u_1(t),u_2(t))\|_{H^1(I)\times L^2(I)} =0,
\ee 
for any bounded interval  $I\subset \RR$.

In particular, no odd, small  breather solutions exist for  (\ref{wave ac}). 
\end{theorem}

By \emph{breather}, we mean   a solution   in $H^1\times L^2$ which is periodic in time, up to the symmetries of the equation. Several integrable equations have \emph{stable} breather solutions, including the sine-Gordon model \cite{Lamb}
\be\label{SG}
\phi_{tt}-\phi_{xx}=-\sin \phi.
\ee
See also \cite{AM,AM1,AM2} and references therein for more details on the literature.

\medskip

Our paper was motivated by (and gives a partial answer to)  a conjecture of Soffer and Weinstein \cite{MR1681113} which says  that no small amplitude breathers should exist for 
\[
\phi_{tt}-\phi_{xx}=-\phi+\phi^3.
\]
Our assumptions  allows other classical examples such as the $\phi^4$ model \cite{MR2282481,MR2318156,MR1402248} 
\[
\phi_{tt}-\phi_{xx}-\phi+\phi^3=0,
\]
the $\phi^6$ model \cite{Lohe},
\[
\phi_{tt}-\phi_{xx}=-\phi+4\phi^3-3\phi^5,
\]
and the sine-Gordon equation \eqref{SG}.

\medskip

Let us discuss the sine-Gordon model \eqref{SG} in more details. 
The energy of  the sine Gordon equation is
\[
E(\phi,  \phi_t)=\int \frac{1}{2} \phi_t^2+\frac{1}{2} \phi_x^2 +(1-\cos \phi).
\] 
Recall that an explicit  family of breathers is known  \cite{Lamb}
\[
B_{\bt}(t,x)= 4\arctan \Big( \frac{\bt}{\al} \frac{\cos(\al t)}{ \cosh(\bt x)}\Big), \qquad \al^2 +\bt^2 =1.
\]
The function $B_{\bt}$  has arbitrarily small energy norm provided $|\bt|$ is small.
Such  breathers oscillate  around the constant, finite energy state $u=0$, which is linearly stable. 
Note that
for all $t$, $B_{\bt}(t)$ is an  even function, which proves in some sense the necessity of the oddness  assumption to obtain a general result of asymptotic stability of the zero solution  such as Theorem~\ref{TH1}.  

\medskip

At a more technical level, the oddness assumption allows to avoid an {\it even} resonance of the linear  operator, supporting the belief that   the existence of breathers is partly related to the presence of resonances. 
However,  the case of the $\phi^4$ model suggests that the situation is more subtle in general.  For this model, in the presence of a resonance, a formal argument due to  Kruskal and Segur \cite{kruskal_segur} excludes the existence of breathers of small amplitude oscillating around $\pm 1$, which are  linearly stable, constant states. As we mentioned above Theorem~\ref{TH1} implies nonexistence of odd breathers of the $\phi^4$ model oscillating about $0$, however our result does not apply for the case considered in \cite{kruskal_segur} since the nonlinearity  written for $u=\phi\pm 1$ is not  odd. 
 
\medskip

For the proof of Theorem \ref{TH1}, we follow a simplified version of the method developed in our recent work \cite{KMM} on asymptotic stability of the kink   under odd perturbations for the $\phi^4$ model.
The main idea in \cite{KMM} and in this paper is the introduction of a generalized \emph{virial identity}, suitably constructed for each considered problem. 
Such approach is inspired by previous works by the second author and Merle \cite{MR1753061,zbMATH01631995,MM3} for the generalized Korteweg-de Vries equations and  by Merle and Rapha\"el \cite{MR2150386} for   nonlinear Schr\"odinger equations. Unlike in those works, here the analysis is performed for solutions in the neighborhood of $0$ and  the spectral analysis reduces to classical operators.
\subsection{Previous results}

Rigorous proofs of nonexistence of breathers solutions in scalar field equations date back to the work by Coron \cite{Coron}. He showed that under the assumption of vanishing energy and $L^\infty$ norm of the solution at infinity in space, and an $L^1_{tx}$ integrability condition, breathers cannot have arbitrarily small periods. Later, Vuillermot \cite{Vuillermot} showed that under some growth and convexity assumptions on the nonlinearity, there are no breather  solutions for suitable scalar field equations. Some formal arguments for existence and nonexistence of breather solutions can be found in Kichenassamy \cite{Kich}, and Birnir-McKean-Weinstein \cite{Birnir}, respectively. Finally, Denzler \cite{denzler} provided a proof of nonexistence of breather  solutions for nonlinearities that are small, complex-analytic perturbations of the sine-Gordon case. From his result, it is concluded that the existence of breathers is a very rare property, probably related to the integrability of the equation.

\medskip

As in the present paper, the nonexistence of breather solutions for \eqref{wave ac} can also be seen  as a consequence of  the asymptotic stability of the vacuum solution in some topology. From this point of view, we refer to the original works of Buslaev and Perelman \cite{bus_per1,bus_per2}, and Buslaev-Sulem \cite{buslaev_sulem} on the NLS case, and Soffer-Weinstein \cite{MR1681113} on nonlinear Klein-Gordon models in 3D. See also the works by Delort \cite{Delort,Delort_fourier}, Lindblad-Soffer \cite{LS1,LS2,LS3}, Bambusi and Cuccagna \cite{Bam_Cucc} and Sterbenz \cite{Ste} on decay of small solutions for nonlinear Klein-Gordon equations in one dimension.  This method has been pushed forward by considering different nonlinearities and lower dimensions. For recent results in this area, see e.g. \cite{MR1664792,Tsai_Yau, Rod_Soffer_Schlag, MR2373326,cuccagna_3, cuccagna_4, MR2835867,Kr_Sch, Bet_Gra_Smets1,Lin_Tao}.

}

\section{Proof of the theorem}\label{VIRIAL}

\medskip

\noindent
{\underline{Step 1: Virial identity.}}
Recall the set of coupled equations \eqref{eqvarphi}.  For a smooth and bounded  function $\psi$ to be chosen later, let
\begin{equation}\label{defI}
\mathcal I (u) := \int \psi (\partial_x u_1) u_2 + \frac 12 \int \psi' u_1 u_2 = \int \left(\psi \partial_x u_1 + \frac 12 \psi' u_1\right) u_2.
\end{equation}
Let $(u_1,u_2)$ be a solution of (\ref{eqvarphi}).  {\YM Since for any $v\in H^1$, $\int (\left(\psi  v_x+\frac{1}{2} \psi' v\right) v=0$, one computes} 
\begin{equation}\label{Ione}
\begin{aligned}
\frac{d}{dt} \mathcal I (u)  & = \int \left(\psi  u_{1,x}+\frac{1}{2} \psi'  u_1\right) u_{1,xx}  
  +  \int \left(\psi  u_{1,x}+\frac{1}{2} \psi'  u_1\right) f(u) \\
& =-\mathcal B(u_1)-\int \psi'\left[ F(u_1)-\frac{1}{2} u_1 f(u_1)\right],
\end{aligned}
\end{equation}
where we have denoted
\[
\mathcal B(u_1)=-\int \left(\psi  u_{1,x}+\frac{1}{2} \psi'  u_1\right) u_{1,xx}=\int \psi' (\partial_x u_1)^2 -\frac{1}{4}\int\psi'''u_1^2.
\]
Therefore,
\begin{equation}
\label{moondawn}
-\frac d{dt} {\mathcal I} =\mathcal B(u_1) + \int \psi'\left[ F(u_1)-\frac{1}{2} u_1 f(u_1)\right].
\end{equation}

\medskip

\noindent
{\underline{Step 2: Coercivity of the bilinear form $\mathcal B$.}}
Now we choose a specific function $\psi$ and we consider the question of the coercivity of the bilinear form $\mathcal B$. 
Let $\lambda>0$ be fixed.   We set  
\begin{equation}\label{defpsi}
\psi(x) : = \lambda\tanh\left(\frac{x}{\lambda}\right),
\end{equation}
in  the definition of $\mathcal I$. 
Note that $\zeta>0$ everywhere. Let $w$ be the following auxiliary function
\be\label{def w}
w:=\zeta u_1;
\quad \zeta(x) :=\sqrt{\psi'(x)}=\sech \left(\frac{x}{\lambda}\right).
\ee
First, note that by integration by parts,
\begin{align*}
\int w_x^2 
&= \int (\zeta \partial_x u_1 +\zeta' u_1  )^2
  = \int \psi' (\partial_x u_1)^2 + 2 \int \zeta \zeta' u_1 (\partial_x v_1) + \int (\zeta')^2 u_1^2 \\
 & = \int \psi' (\partial_x u_1)^2  - \int \zeta \zeta'' u_1^2 \\
 & = \int \psi' (\partial_x u_1)^2  - \int\frac{\zeta''}{\zeta} w^2.
\end{align*}
Thus,
\begin{equation}\label{idgrad}
\int \psi' (\partial_x u_1)^2 = \int w_x^2  +  \int\frac{\zeta''}{\zeta} w^2.
\end{equation}
Second,
$$
\int\psi'''u_1^2 = \int \frac{(\zeta^2)''}{\zeta^2} w^2
= 2 \int \left(\frac {\zeta''}{\zeta}+ \frac{(\zeta')^2}{\zeta^2} \right) w^2.
$$
Therefore,
\begin{equation}\label{change}\begin{aligned}
\mathcal B(u_1) &= \int \psi' (\partial_x u_1)^2 -\frac{1}{4}\int\psi'''u_1^2\\
& = \int w_x^2 + \frac 12  \int \left( \frac {\zeta''}{\zeta} -\frac{(\zeta')^2}{\zeta^2} \right) w^2.
\end{aligned}\end{equation}
Set
\[
{\mathcal B}^\sharp(w):= \int \left(w_x^2 - V w^2\right),\quad \hbox{where} \quad
V:= - \frac 12    \left( \frac {\zeta''}{\zeta} -\frac{(\zeta')^2}{\zeta^2} \right),
\]
so that 
\be\label{B sharp}
 {\mathcal B}^\sharp(w)=\mathcal B(u_1).
\ee
Note that by \eqref{defpsi} and direct computations,
\be\label{V}
V(x)=
\frac{1}{2\lambda^2}\sech^2\left(\frac{x}{\lambda}\right).
\ee
{\YM
Recall that the index of the operator 
$
-\frac{d^2}{dx^2}-V
$
associated to $\mathcal B^\sharp$ is  $1$, which means  that this operator has only one  negative, discrete eigenvalue with an even   corresponding eigenfunction. This follows from the well known fact (see e.g. Titchmarsh \cite[\S 4.19]{Tit} and \cite[p. 55]{Goldman})  that for any $\lambda>0$, the index $\kappa$ of the operator 
\[
-\frac{d^2}{dx^2}-\frac{V_0}{\lambda^2}\sech^2\left(\frac{x}{\lambda}\right), \quad V_0\in [0,+\infty),
\]
is the largest integer  such that 
\[
\kappa<\frac{1}{2}\sqrt{4 V_0+1}+\frac{1}{2}.
\]
Here, our claim concerning $\mathcal B^\sharp$ follows from
$
1< \frac{\sqrt{3}+1}{2}<2.
$
In fact, we claim the following more precise result.
 \begin{lemma}\label{le:posVir}
For any $\lambda>0$ and   for any odd function $w\in H^1$, 
\begin{equation}\label{posVir1}
    {\mathcal B}^\sharp(w)
  \geq \frac{3}{4} \int w_x^2.
\end{equation}
\end{lemma}

\begin{proof} 
We write 
\begin{equation}
\mathcal B^\sharp(w)=\frac 34 \int w_x^2+\frac 14 \int \left(w_x^2 -  \frac{2}{\lambda^2}\sech^2\left(\frac{x}{\lambda}\right) w^2\right).
\label{decomp pos Vir}
\end{equation}
Since $ \frac{1}{2}\sqrt{4.2+1}+\frac{1}{2}=2$,
the second term in the right-hand side above is nonnegative for any odd function $w$, and the result follows.
\end{proof}}

\medskip

\noindent
{\underline{Step 3: Control of the error terms and conclusion of the Virial argument.}}
{\YM As in statement of Theorem \ref{TH1}, we consider an odd   solution $(u_1(t),u_2(t))$ of \eqref{eqvarphi}   global in $H^1\times L^2$ for $t\geq 0$ and satisfying, for all $t\geq 0$,
\begin{equation}\label{smallphi}
 \|(u_1(t),u_2(t))\|_{H^1\times L^2} < \varepsilon.
\end{equation}
We define}
\begin{equation}\label{nloc}
\|u_1\|_{H^1_{\omega}}^2 :=
\int \left( |\partial_x u_1|^2 + u_1^2   \right) \sech\left(x\right),\quad
\|u_2\|_{L^2_{\omega}}^2 :=
\int  u_2^2   \sech\left(x\right),
\end{equation}
and
\begin{equation}\label{nloc2}
\|(u_1,u_2)\|_{H^1_{\omega}\times L^2_{\omega}}^2 :=\|u_1\|_{H^1_{\omega}}^2 + \|u_2\|_{L^2_{\omega}}^2.
\end{equation}

The key ingredient of the proof of asymptotic stability in the energy space is the following result.
\begin{proposition}\label{pr:11}
For $\varepsilon>0$ small enough,
\begin{equation}\label{11un}
\int_{0}^{+\infty}  \|(u_1(t),u_2(t))\|_{H^1_{\omega}\times L^2_{\omega}}^2  dt 
\lesssim \varepsilon^2 .
\end{equation}
\end{proposition}

\begin{proof}[Proof of Proposition \ref{pr:11}]
We consider the virial-type quantity $\mathcal I(t)$   defined in \eqref{defI} with $\lambda=100$.

\medskip

The proof of \eqref{11un} is based on  the following two estimates, which hold for some  constant  $C_1, C_2>0$:
\begin{align}
  -\frac d{dt} \mathcal I  
&   \geq   C_1   \|u_1\|_{H^1_{\omega}}^2 ,  
\label{10un} \\
  \frac d{dt} \int \sech\left(x\right) u_1 u_2 
&\geq    \|u_2\|_{L^2_{\omega}}^2
-C_2  \|u_1\|_{H^1_{\omega}}^2 .
\label{10trois}
\end{align}

\medskip
{\YM 
First we prove   \eqref{11un} assuming \eqref{10un}   and \eqref{10trois}.
Integrating \eqref{10un} on $[0,t_0]$, using the bound \eqref{smallphi}, and passing to the limit as $t_0\to +\infty$, we find  
\begin{equation}\label{11unbis}
\int_{0}^{+\infty}  \|u_1(t)\|_{H^1_{\omega}}^2  dt 
\lesssim \varepsilon^2 .
\end{equation}
Then, using 
\eqref{10trois} similarly, we obtain \eqref{11un}.}

\medskip

Thus, to finish the proof, we only have to prove \eqref{10un} and \eqref{10trois}. We begin with the proof  of \eqref{10un}. We have from \eqref{moondawn} and \eqref{B sharp},
{\YM
\[
-\frac d{dt} {\mathcal I} = \mathcal B (u_1) -  \int \psi'\left[ F(u_1)-\frac{1}{2} u_1 f(u_1)\right] .
\]
Recall from Section \ref{VIRIAL} the notation $w = u_1 \zeta=u_1 \sech\left(\frac x{100}\right)$ (see \eqref{def w}).
 From Lemma \ref{le:posVir}, we have
\[
\mathcal B (u_1)=  \mathcal B^\sharp(w) \geq \frac 34 \int w_x^2 \quad \text{and equivalently}\quad
 \int w_x^2 \geq  \frac {1}{5.10^3} \int    \sech^2\left(\frac{x}{100}\right) w^2.
\]
Thus,  
\begin{equation}\label{h1}
\mathcal B (u_1)\gtrsim\|\partial_x w\|_{L^2}^2\gtrsim
\int  \sech^2\left(\frac {x}{100}\right) w^2=
\int \sech^4\left(\frac {x}{100}\right) u_1^2\gtrsim
\int \sech\left({x}\right) u_1^2 .
\end{equation}
Next, we have
\begin{align*}
\|\partial_x w\|_{L^2}^2 &\gtrsim \int \sech^2\left(\frac {x}{100}\right) |\partial_x w|^2 
  = \int \sech^2\left(\frac {x}{100}\right) |\zeta \partial_x u_1 + \zeta' u_1 |^2\\
& \gtrsim \int \sech^4\left(\frac {x}{100}\right) |\partial_x u_1|^2
+ 2 \int \sech^2\left(\frac {x}{100}\right) \zeta \zeta' (\partial_x u_1) u_1 + \int \sech^2\left(\frac {x}{100}\right) (\zeta')^2 u_1^2\\
& \gtrsim \int \sech\left(x\right)|\partial_x u_1|^2
+ \int u_1^2\left( - \left(\sech^2\left(\frac {x}{100}\right) \zeta \zeta'\right)' + \sech^2\left(x\right) (\zeta')^2\right).
\end{align*}
Thus, using \eqref{h1},
\[
\int \sech (x )|\partial_x u_1|^2
\lesssim \|\partial_x w\|_{L^2}^2+\int \sech^2 (x ) u_1^2
\lesssim \|\partial_x w\|_{L^2}^2 \lesssim \mathcal B (u_1).
\]
This implies
\begin{equation}\label{rwv}
 \mathcal B (u_1)\gtrsim\|\partial_x w\|_{L^2}^2 \gtrsim \|u_1\|_{H^1_{\omega}}^2.  
\end{equation}}

{\YM Now, we claim  that for any $q>0$,
\begin{equation}\label{SFcinq}
\int  \psi' |u_1|^{2+q} \lesssim \|u_1\|_{L^\infty}^q \|\partial_x w\|_{L^2}^2 \lesssim \varepsilon^q \|\partial_x w\|_{L^2}^2.
\end{equation}
Indeed, by parity, the definition of $\psi$ \eqref{defpsi} and $w$ \eqref{def w}, we have (with $\lambda=100$)
\begin{align*}
\int  \psi' |u_1|^{2+q}
\lesssim   \int_0^{+\infty} e^{-\frac {2 x}{{\lambda }}} |u_1|^{2+q}
\lesssim   \int_0^{+\infty} e^{q \frac x{\lambda}} |w|^{2+q}.
\end{align*}
Integrating by parts and using $w(0)=0$ (the function $w$ is odd)
\begin{align*}
\int_0^{+\infty} e^{q \frac x{\lambda}} |w|^{2+q}
&= - \frac{\lambda}q \int_0^{+\infty} e^{q \frac x{\lambda}} \partial_x(|w|^{2+q})
= - \frac {2+q}q \lambda \int_0^{+\infty} e^{q\frac x{\lambda}} (\partial_x w) w |w|^q\\
&\leq C \|u_1\|_{L^\infty}^{\frac q2} \int_0^{+\infty} e^{q \frac x{2 \lambda}} |\partial_x w| |w|^{1+\frac q2}
\leq C^2 \|u_1\|_{L^\infty}^q \int_0^{+\infty} |\partial_x w|^2 +\frac 14  \int_0^{\infty} e^{q\frac x{\lambda}} |w|^{2+q}.
\end{align*}
Thus,
\[
\int_0^{+\infty} e^{q \frac x{\lambda}} |w|^{2+q}
\lesssim  \|u_1\|_{L^\infty}^q \|\partial_x w\|_{L^2}^2,
\]
and \eqref{SFcinq} is proved.  Since, for some $p>1$,
\[
\left|  F(u_1)-\frac{1}{2} u_1 f(u_1) \right |\lesssim |u_1|^{p+1},
\]
we   estimate  by \eqref{SFcinq},
\begin{equation}\label{SFsix}
\left|\int \psi'\left[ F(u_1)-\frac{1}{2} u_1 f(u_1)\right]\right| \lesssim \varepsilon^{p-1} \|\partial_x w\|_{L^2}^2.
\end{equation}}
Combining \eqref{rwv} and \eqref{SFsix} proves \eqref{10un} for $\ve$ small enough.


Finally, we show  \eqref{10trois}. From \eqref{eqvarphi}, we compute
\[
\begin{aligned}
\frac{d}{dt}\int \sech(x) u_1 u_2&=\int\sech(x) u_2^2-\int \sech(x) u_{1,x}^2+\int \left[a\sech(x)+\frac{1}{2}\sech''(x)\right] u_1^2 + \int \sech(x) u_1 f(u_1).
\end{aligned}
\]
From this,  \eqref{10trois} follows readily by definition of the norm in $H^1_\omega\times L^2_\omega$ and (\ref{SFsix}).
This ends the proof of  Proposition \ref{pr:11}.
\end{proof}

\medskip
\noindent
{\underline{Step 4: Conclusion of the proof of  Theorem \ref{TH1}.}} 
Let
\begin{equation}\label{defH}
\mathcal H(t):= \int \sech\left(x\right)\left[u_{1,x}^2 +u_1^2 + u_2^2\right](t) .
\end{equation}
Then, using \eqref{eqvarphi}, we have {\YM
\begin{equation}\label{Kone}\begin{aligned}
\frac{d}{dt}{\mathcal  H}  &= 2 \int  \sech\left(x\right) \left(u_{1,xt} u_{1,x} + u_{1,x} u_1  + u_{2,t} u_2\right)\\
& = 2 \int \sech\left(x\right) \left[ u_{2,x} u_{1,x} +  u_2 u_1  + \left(u_{1,xx} +\mm u_1+f(u_1)\right)u_2 \right]\\
& =  2 \int  \sech\left(x\right) \left[\left(1+\mm\right)u_1+f(u_1)\right]u_2 - 2 \int \sech'(x) u_2 u_{1,x} ,
\end{aligned}\end{equation}
 and so
\begin{equation}\label{Ktwo}
\left| \frac{d}{dt}{\mathcal  H} \right| 
\lesssim \int \left(u_{1,x}^2+u_1^2+{u_2^2} \right) \sech\left(x\right)   \lesssim  \|u(t)\|_{H^1_{\omega}\times L^2_{\omega}}^2. 
\end{equation}}

From \eqref{11un} there exists a sequence $t_n\to +\infty$ such that $\mathcal H(t_n)\to 0$.
Let $t\in \RR$. Integrating on $[t,t_n]$ and passing to the limit as $n\to +\infty$ we obtain
\[
\mathcal H(t) \lesssim \int_{t}^{+\infty} \|u(t)\|_{H^1_{\omega}\times L^2_{\omega}}^2 dt.
\]
From \eqref{11un} it follows that $\lim_{t\to +\infty} \mathcal H(t)=0.$  
Thus, $\lim_{t\to + \infty} \|u\|_{H^1_{\omega}\times L^2_{\omega}}=0.$ 
 This implies (\ref{Conclusion_0}) and finishes the proof of the Theorem.


\providecommand{\bysame}{\leavevmode\hbox to3em{\hrulefill}\thinspace}
\providecommand{\MR}{\relax\ifhmode\unskip\space\fi MR }
\providecommand{\MRhref}[2]{%
  \href{http://www.ams.org/mathscinet-getitem?mr=#1}{#2}
}
\providecommand{\href}[2]{#2}

\end{document}